\documentclass{article}
\usepackage{subfig}
\usepackage[section] {placeins}
\usepackage{amsthm}
\usepackage{amssymb,amsmath, mathdots}
\usepackage{anysize}
\usepackage{psfrag}
\usepackage{url}
\newtheorem{theorem}{Theorem}[section]
\newtheorem{definition}{Definition}[section]

\newtheorem{rem}{Remark}[section]

\newtheorem{cor}{Corollary}[section]
\newtheorem{conj}{Conjecture}[section]

\numberwithin{table}{section}
\begin{document}

\title{The Prevalence of Persistent Tangles.}

\author{Louis H. Kauffman\\
Department of Mathematics, Statistics and Computer Science \\ 851 South Morgan Street   \\ University of Illinois at Chicago\\
Chicago, Illinois 60607-7045 USA\\ and\\ Department of Mechanics and Mathematics\\ Novosibirsk State University\\Novosibirsk, Russia \\
\texttt{kauffman@uic.edu}\\
and\\
Pedro Lopes\\
Center for Mathematical Analysis, Geometry, and Dynamical Systems, \\
Department of Mathematics, \\
Instituto Superior T\'{e}cnico, Universidade de Lisboa\\
1049-001 Lisbon, Portugal \\
\texttt{pelopes@math.tecnico.ulisboa.pt}}
\maketitle

\begin{abstract}
This article addresses persistent tangles. These are tangles whose presence in a knot diagram forces that diagram to be knotted. We provide  new methods for constructing persistent tangles. Our techniques rely mainly on the existence of non-trivial colorings for the tangles in question.  Our main result in this article is that any knot admitting a non-trivial coloring gives rise to persistent tangles. Furthermore, we discuss when these persistent tangles are  non-trivial.
\end{abstract}



\section{Introduction}\label{sec:intro}

This article addresses the notion of  \emph{persistent tangle}, by which we mean, a tangle whose appearance in a knot diagram forces that diagram to be knotted. We show that persistent tangles are prevalent, as subtangles, in diagrams  of non-trivially colored knots.
This article also addresses the following issue: local features that provide global information. For instance, we have in mind the identification of entanglement in long polymers or DNA. The size of these long molecules complicates the identification of entanglement (global information). Therefore, the recognition of persistent tangles (local feature) should be relevant in this context. The techniques in the proofs are mainly elaborations of the following idea: we endow our tangles with specific non-trivial colorings  that assign the same color to the start- and end-points of the tangle, over an appropriate modulus, see \cite{SilverWilliams}. This coloring can be extended (monochromatically) to the rest of the knot diagram it may belong to. Thus, that diagram is non-trivially colored, and therefore knotted. We often use Fox colorings, but not exclusively. In a Fox coloring the colors are in $\mathbf{Z}/N\mathbf{Z}$ for an appropriate positive integer $N$ (or simply in $\mathbf{Z}$) and the sum of the colors assigned to the undercrossing arcs at a crossing is twice the color assigned to the over crossing arc \cite{lhKauffman}. The term \emph{knot} in this article  means a $1$-component link.

Along with giving new constructions for persistent tangles, we also formulate a conjecture that ``irreducible tangles are persistent'' (see Section \ref{sec:results} for the definition of irreducibility and the precise statement of this conjecture). Solution to this conjecture appears to require techniques beyond the reach of the present paper.

\subsection{Acknowledgements.}\label{subsec:ack}

Kauffman's work was supported by the Laboratory of Topology and Dynamics, Novosibirsk State University (contract no. 14.Y26.31.0025 with the Ministry of Education and Science of the Russian Federation).

Lopes acknowledges support from FCT (Funda\c c\~ao para a Ci\^encia e a Tecnologia), Portugal, through project FCT PTDC/MAT-PUR/31089/2017, ``Higher Structures and Applications''.

\section{First  results}\label{sec:1st-results}

We now recall the basic definitions and identify the original  persistent tangle and the trivial ones.

\begin{definition}
A coloring of a knot $K$ by a quandle $X$ is a homomorphism from the fundamental quandle of the knot $K$ to the quandle $X$ \cite{SMatveev, DJoyce}. A non-trivial coloring is one such homomorphism whose range is non-singular.
\end{definition}

\begin{rem}\label{thm:non-trivialcols.}
The unknot does not admit non-trivial colorings. Note however that unlinks can sometimes admit non-trivial colorings.
\end{rem}
\begin{proof}
The standard diagram of the unknot is a circle on the plane without self-intersections. Therefore, the colorings it admits only involves one color.
\end{proof}
\begin{rem}
We will use Remark \ref{thm:non-trivialcols.} in the following form. If a knot admits non-trivial colorings then it is non-trivial.
\end{rem}

\begin{definition}\label{def:tangle}
A tangle  is an embedding of one $($respect., two$)$ arc$($s$)$ in a ball with the fixed end points on the surface of the ball. Two tangles are equivalent if they are related by an ambient isotopy, keeping the endpoints fixed. The diagrammatical counterpart is a piece of knot diagram on a disc, with the endpoints on the boundary of the disc. The Reidemeister moves are restricted to the disc with the endpoints fixed on the boundary of the disc. See Figures \ref{fig:krebes}, \ref{fig:noname}, \ref{fig:persist-1-arc}, and \ref{fig:persist-1-arc-bis}, for illustrating examples. In a subsequent article we will look into the $n$-tangle case with $n>2$.
\end{definition}

\begin{definition}\label{def:persistent-tangle}
A persistent tangle is a  tangle whose presence in a knot diagram implies this knot is non-trivial. Figure \ref{fig:noname} provides an example of a persistent tangle.
\end{definition}

The original  persistent tangle is due to Krebes \cite{Krebes} and is depicted in Figure \ref{fig:krebes}, see also Figure \ref{fig:noname}. Krebes proved   persistence of tangles like those depicted in Figures \ref{fig:krebes} and \ref{fig:noname} by way of the bracket polynomial. Later Silver and Williams proved the same sort of result by way of Fox colorings \cite{SilverWilliams}. Our approach here is in the spirit of \cite{SilverWilliams} but we consider other colorings besides the Fox colorings. We acknowledge also the work of other authors in related matters \cite{SMAbernathy, KrebesSilverWilliams, KauffmanGoldman, PrzytyckiSilverWilliams, Ruberman, SilverWilliams2}. In \cite{KauffmanGoldman} invariants of knots and tangles are formulated via sums over weighted trees in the same way as the more recent paper by Silver and Williams, \cite{SilverWilliams2}, and it is shown how the tangle fraction and some generalizations arise by using the checkerboard graph for knots and links. Furthermore, \cite{KauffmanGoldman} interprets this combinatorics in terms of
the current flow in electrical circuits. It is possible that there is more work to be done about persistent tangles in this domain.

\begin{figure}[!ht]
	\psfrag{1}{\huge$1$}
	\psfrag{2}{\huge$2$}
	\psfrag{3}{\huge$3$}
	\psfrag{...}{\huge$\cdots$}
	\psfrag{p xings}{\huge$\text{$p$-xings}$}
	\psfrag{4}{\huge$4$}
	\psfrag{p-2}{\huge$p-2$}
	\psfrag{p-1}{\huge$p-1$}
	\psfrag{0}{\huge$0$}
	\centerline{\scalebox{.4}{\includegraphics{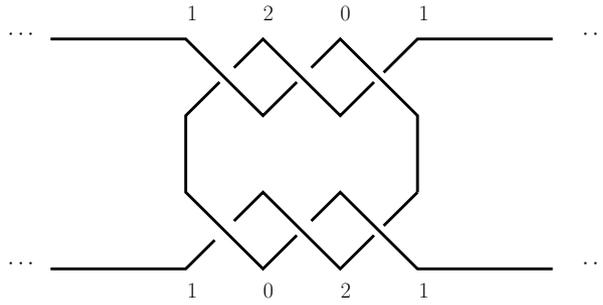}}}
	\caption{The original persistent tangle, by Krebes. The non-trivial coloring at issue uses the dihedral quandle of order $3$.}\label{fig:krebes}
\end{figure}

Note that the Krebes example is not a rational tangle. In fact, no persistent tangle can be rational \cite{HenrichKauffman} since any rational tangle can be inserted into an unknot \cite{LambropoulouKauffman}. The reader may enjoy proving this as an exercise.

\begin{figure}[!ht]
	\psfrag{1}{\huge$1$}
	\psfrag{2}{\huge$2$}
	\psfrag{3}{\huge$3$}
	\psfrag{...}{\huge$\cdots$}
	\psfrag{p xings}{\huge$\text{$p$-xings}$}
	\psfrag{4}{\huge$4$}
	\psfrag{p-2}{\huge$p-2$}
	\psfrag{p-1}{\huge$p-1$}
	\psfrag{0}{\huge$0$}
	\centerline{\scalebox{.4}{\includegraphics{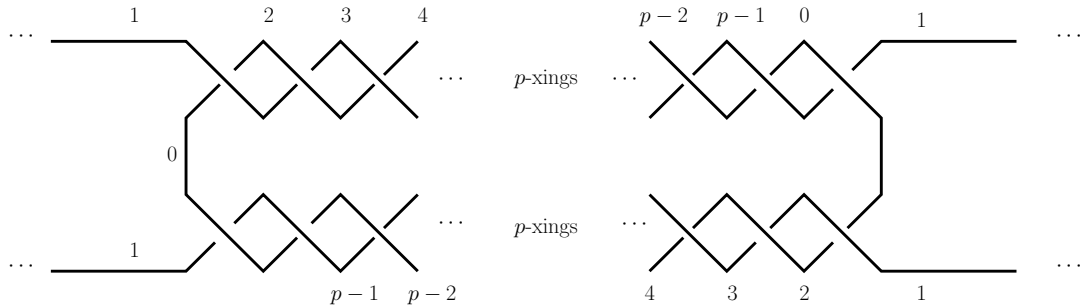}}}
	\caption{A persistent tangle. The non-trivial coloring at issue uses the dihedral quandle of order $p$, for an odd prime $p$.}\label{fig:noname}
\end{figure}

We now elaborate on a number of constructions that obviously give rise to persistent tangles. We call these trivial persistent tangles. We start with the case of a $1$-tangle.
We recall that  \emph{genus of a knot}  is the least genus of the oriented surfaces whose boundary is the knot at issue.

\begin{theorem}\label{thm:genus}
Genus is additive under connected sums of knots. A knot is trivial if and only if its genus is $0$.
\end{theorem}
\begin{proof}
These are known results. See \cite{Lickorish} for a proof.
\end{proof}
\begin{cor}\label{cor:genus}
A non-trivial knot gives rise to a persistent $1$-tangle by disconnecting any of its arcs in one of its diagrams. This is our first instance of a trivial persistent tangle.
\end{cor}
\begin{proof}
Attaching to our $1$-tangle a second $1$-tangle amounts to performing a connected sum of a non-trivial knot with another knot. Applying Theorem \ref{thm:genus} concludes the proof.  See Figure \ref{fig:persist-1-tangle}, disregarding colorings.
\end{proof}

\begin{cor}\label{cor:persistent-colored-1-tangles}
Every knot admitting a non-trivial coloring yields a persistent $1$-tangle via disconnecting any one of its arcs.
\end{cor}
\begin{proof}
If the knot admits a non-trivial coloring then it is non-trivial. We can now apply Corollary \ref{cor:genus} to conclude the proof. See Figure \ref{fig:persist-1-tangle}, again.
\end{proof}
\begin{figure}[!ht]
	\psfrag{1}{\huge$1$}
	\psfrag{2}{\huge$2$}
	\psfrag{3}{\huge$3$}
	\psfrag{...}{\huge$\cdots$}
	\psfrag{p xings}{\huge$\text{$p$-xings}$}
	\psfrag{4}{\huge$4$}
	\psfrag{p-2}{\huge$p-2$}
	\psfrag{p-1}{\huge$p-1$}
	\psfrag{0}{\huge$0$}
	\centerline{\scalebox{.4}{\includegraphics{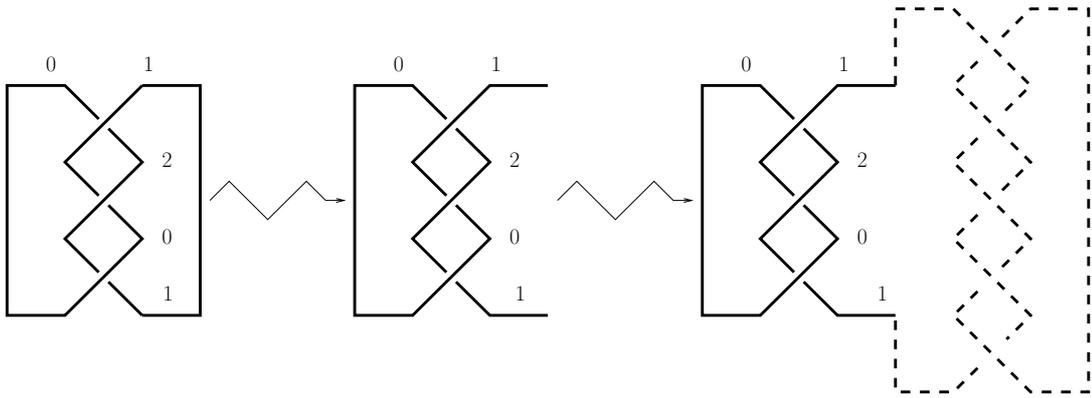}}}
	\caption{Left hand-side: trefoil with a non-trivial coloring - the non-trivial knot. Middle: The non-trivial knot converted into a $1$-tangle by disconnecting  one arc. Right hand-side: Identification of the $1$-tangle in a new knot implies new knot is non-trivial.}\label{fig:persist-1-tangle}
\end{figure}

Corollary \ref{thm:persistent-2-tangles} is a different view of the fact expressed in Corollary \ref{cor:persistent-colored-1-tangles} yet paving the way for the subsequent material.

\begin{cor}\label{thm:persistent-2-tangles}
Assume $K$ is a knot admitting non-trivial colorings. Then $K$ gives rise to a persistent $2$-tangle by cutting an arc at two distinct points.
\end{cor}
\begin{proof}
Since $K$ admits a non-trivial coloring, there exists a diagram $D$ of $K$ which supports such a coloring. We disconnect one arc at two distinct points thereby producing a tangle with two start-points and two end-points, all of them receiving the same color, say $a$. Clearly, if this tangle is found in another knot diagram, this knot diagram is non-trivially colored. The arcs of the new diagram which do not belong to the tangle are monochromatically colored with $a$; the tangle part of the new diagram is colored as in the original knot diagram.   Figure \ref{fig:persist-1-arc} provides an illustration of this process.
\begin{figure}[!ht]
	\psfrag{0}{\huge$0$}
	\psfrag{1}{\huge$1$}
	\psfrag{2}{\huge$2$}
	\centerline{\scalebox{.4}{\includegraphics{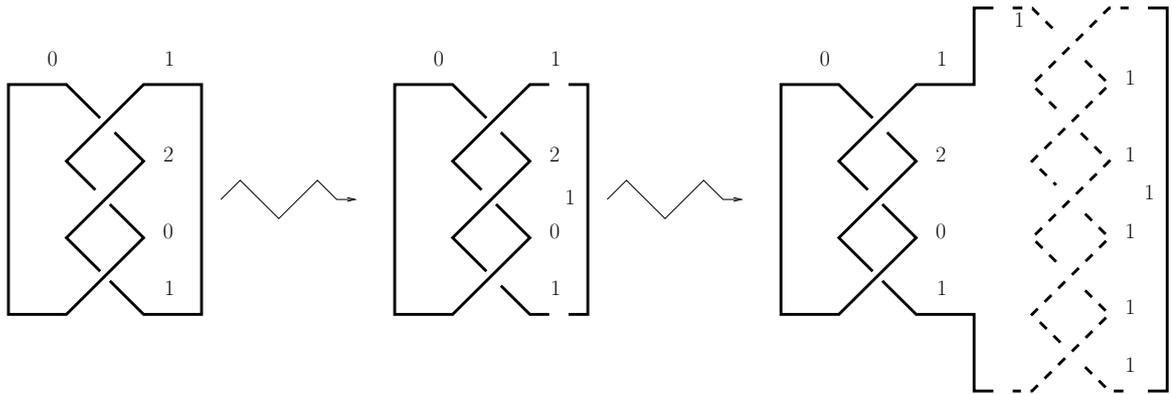}}}
	\caption{Left-hand side: a non-trivial coloring (Fox tricoloring) on a knot (trefoil) diagram. Middle: disconnecting one of the arcs at two points, producing a persistent tangle. Right-hand side: The persistent tangle inside a new non-trivial knot. This example produces a connected sum of knots. Note the difference from Figure \ref{fig:persist-1-tangle}}\label{fig:persist-1-arc}
\end{figure}
\end{proof}

The current article is an extension of these results, especially Corollary \ref{thm:persistent-2-tangles}, but we will disconnect two distinct arcs of (certain) knot diagrams instead of one, in order to produce persistent $2$-tangles.

\section{Results}\label{sec:results}

The next result advertises the possibility for the persistent tangles to be found inside knot diagrams that are not connected sums. In this way, we hope to enlarge the variety of persistent tangles. In particular we hope to obtain persistent tangles which give rise to knots which are not connected sums.

\begin{theorem}\label{thm:1st-theorem}
If a knot $K$ admits a non-trivial coloring over a diagram $D$ with distinct arcs bearing the same color, then it gives rise to a persistent tangle by cutting at one point each of the two arcs that have the same color.
\end{theorem}
\begin{proof}
We distinguish two situations.
\begin{enumerate}
\item In the first one, the two arcs receiving the same color are both sides to the same face of the diagram, see Figure \ref{fig:persist-2-tangle} and Figure \ref{fig:persist-1-arc-bis} (the  two tangle diagrams). We then disconnect each of the arcs referred to, thereby obtaining a tangle. This tangle can be non-trivially colored, since it stems from a knot that can be non-trivially colored. Thus, if this tangle is found in a new knot diagram, the rest of this diagram can be monochromatically colored. The tangle obtained is thus a persistent tangle. This concludes the proof in this situation.

\item In the second situation, the two arcs receiving the same color are sides to different faces of the diagram. In this case we perform a finite number of type II Reidemeister moves in order to bring one of the arcs over to the vicinity of the other, see Figure  \ref{fig:persist-1-arc-bis} (the  two knot diagrams). Also note that the new coloring obtained by consistently recoloring after each of the type II Reidemeister moves is a non-trivial coloring, since recoloring after Reidemeister moves (colored Reidemeister moves) preserves non-triviality, \cite{pLopes}. Now there are two arcs bearing the same color and both are sides to the same face of the diagram. We can then apply the reasoning in $1.$ to conclude the proof in this situation.
\end{enumerate}
The proof is complete.
\end{proof}
\begin{figure}[!ht]
	\psfrag{a1}{\huge$a_1$}
	\psfrag{a2}{\huge$a_2$}
	\psfrag{A1}{\huge$A_1$}
	\psfrag{A2}{\huge$A_2$}
	\psfrag{A1'}{\huge$A'_1$}
	\psfrag{A2'}{\huge$A'_2$}
	\centerline{\scalebox{.4}{\includegraphics{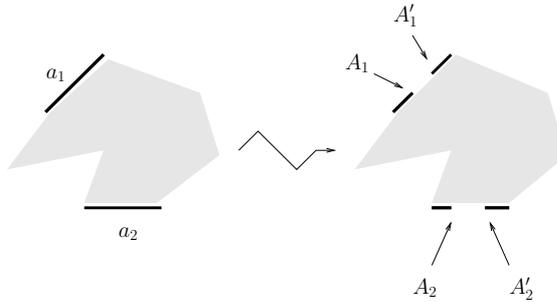}}}
	\caption{Left hand-side: the knot diagram equipped with a non-trivial coloring (not shown) with arcs $a_1$ and $a_2$ that receive the same color. Right hand-side: the $2$ tangle obtained by disconnecting arcs $a_1$ and $a_2$; $A_1, A'_1, A_2, A'_2$ are start- and/or end- points after disconnecting arcs $a_1$ and $a_2$. The shadowed regions do not contain any arc nor crossing of the diagram.}\label{fig:persist-2-tangle}
\end{figure}
\begin{figure}[!ht]
	\psfrag{0}{\huge$0$}
	\psfrag{1}{\huge$1$}
	\psfrag{2}{\huge$2$}
	\psfrag{...}{\huge$\cdots$}
	\centerline{\scalebox{.3}{\includegraphics{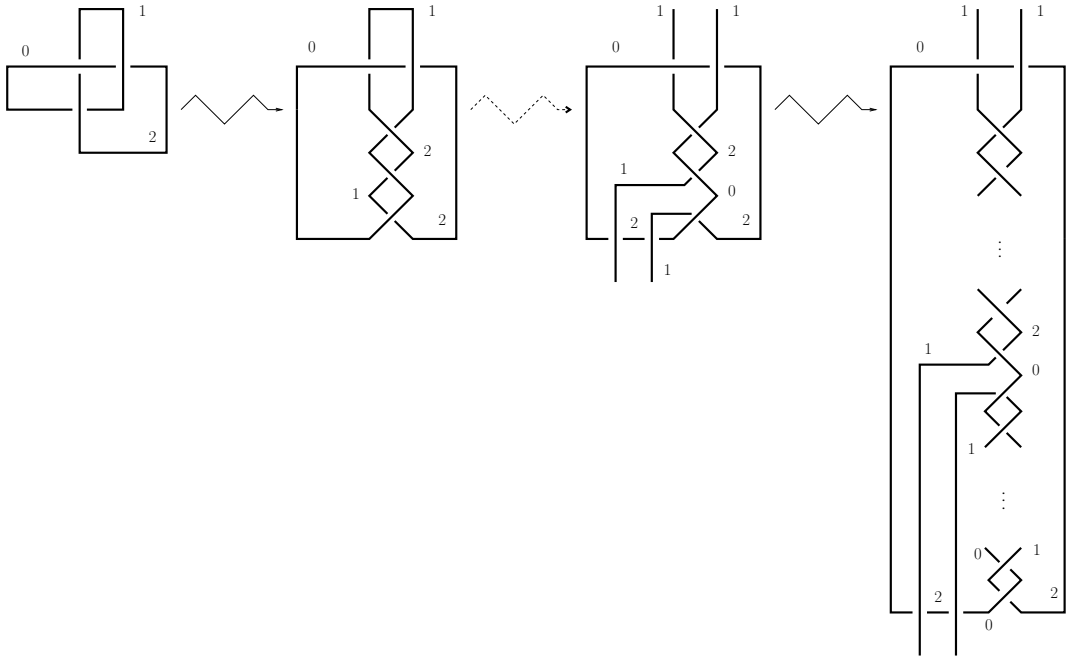}}}
	\caption{We enumerate the diagrams in this Figure $1$ though $4$, starting from the left-most. $1$: the diagram of a trefoil that evolves into ($2$) an equivalent diagram via a type II Reidemeister move. $3$. Another type II Reidemeister move and we obtained two arcs bearing the same color and adjacent to the same face. The appropriate disconnections are performed, thereby obtaining a tangle. $4$: (type II) Reidemeister moves are performed on the tangle (endpoints are kept fixed).}\label{fig:persist-1-arc-bis}
\end{figure}
We remark that in spite of the conditions of Theorem \ref{thm:1st-theorem} being satisfied, we may end up with unexpected outputs, like unlinked components, see Figure \ref{fig:8-16}.

\begin{figure}[!ht]
	\psfrag{8-16}{\huge$\mathbf{8_{16}}$}
	\psfrag{0}{\huge$0$}
	\psfrag{1}{\huge$1$}
	\psfrag{2}{\huge$2$}
	\psfrag{3}{\huge$3$}
	\psfrag{4}{\huge$4$}
	\psfrag{5}{\huge$\mathbf{5}$}
	\psfrag{d}{\huge$d$}
	\centerline{\scalebox{.4}{\includegraphics{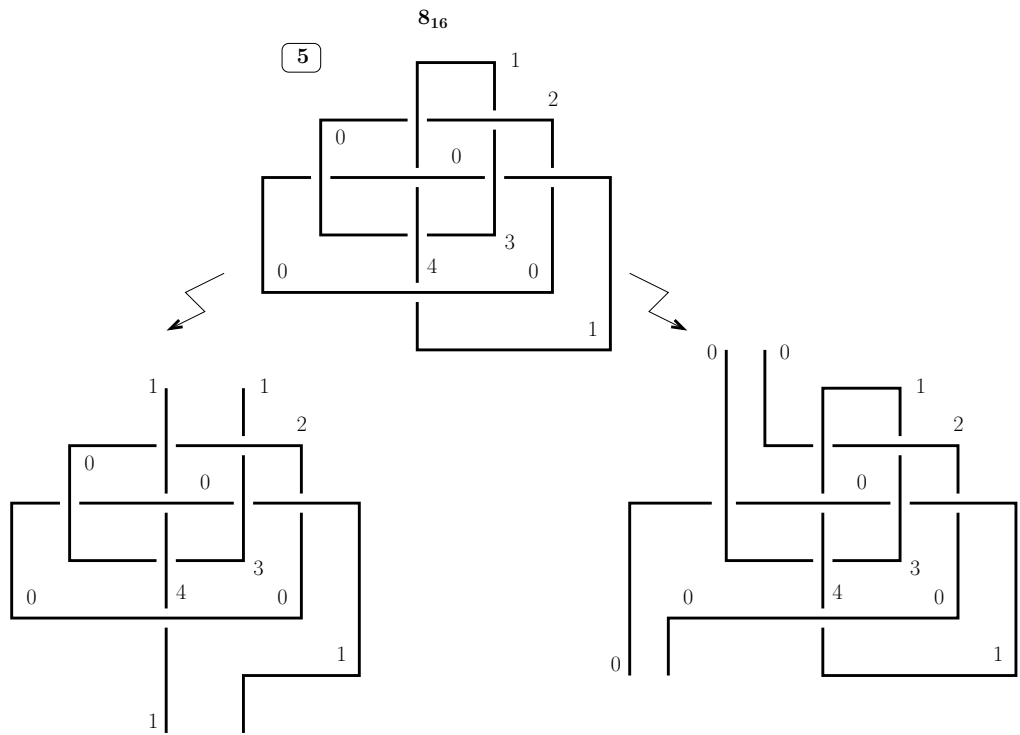}}}
	\caption{Top part of the Figure: knot $8_{16}$ endowed with a non-trivial $5$-coloring (flagged by the boxed $5$ on the top left). There are distinct arcs receiving the same colors. In the bottom part of the Figure we break pairs of arcs that receive the same color in order to produce persistent tangles, a la Theorem \ref{thm:1st-theorem}. Interesting feature with the persistent tangles. Bottom left: the arcs are (independently) unknotted; bottom right: one of the arcs is unknotted (per se), the other is not.}\label{fig:8-16}
\end{figure}

In Corollary \ref{cor:linking} we give a sufficient condition for the output to display linking among the components.

\begin{cor}\label{cor:linking}
Assume $K$ is a knot admitting a non-trivial coloring over a diagram such that two distinct arcs receive the same color. Then there is an equivalent diagram of $K$ with two arcs receiving the same color and such that cutting at one point each of the two arcs that have the same color yields a tangle with non-zero linking number.
\end{cor}
\begin{proof}
Look at Figure \ref{fig:6-2}. The right number of Type II Reidemeister moves increases the linking number while preserving the desired color on the arcs to be disconnected.
\end{proof}

\begin{figure}[!ht]
	\psfrag{0}{\huge$0$}
	\psfrag{1}{\huge$1$}
	\psfrag{2}{\huge$2$}
	\psfrag{3}{\huge$3$}
	\psfrag{4}{\huge$4$}
	\psfrag{5}{\huge$5$}
	\psfrag{6}{\huge$6$}
	\psfrag{7}{\huge$7$}
	\psfrag{8}{\huge$8$}
	\psfrag{9}{\huge$9$}
	\psfrag{10}{\huge$10$}
	\psfrag{6k-9(k-1)}{\huge$6k-9(k-1)$}
	\psfrag{6(k+1)-9k}{\huge$6(k+1)-9k$}
	\psfrag{11}{\huge$\mathbf{11}$}
	\psfrag{6-2}{\huge$6_2$}
	\psfrag{k}{\huge$k$}
	\psfrag{...}{\huge$\cdots$}
	\centerline{\scalebox{.4}{\includegraphics{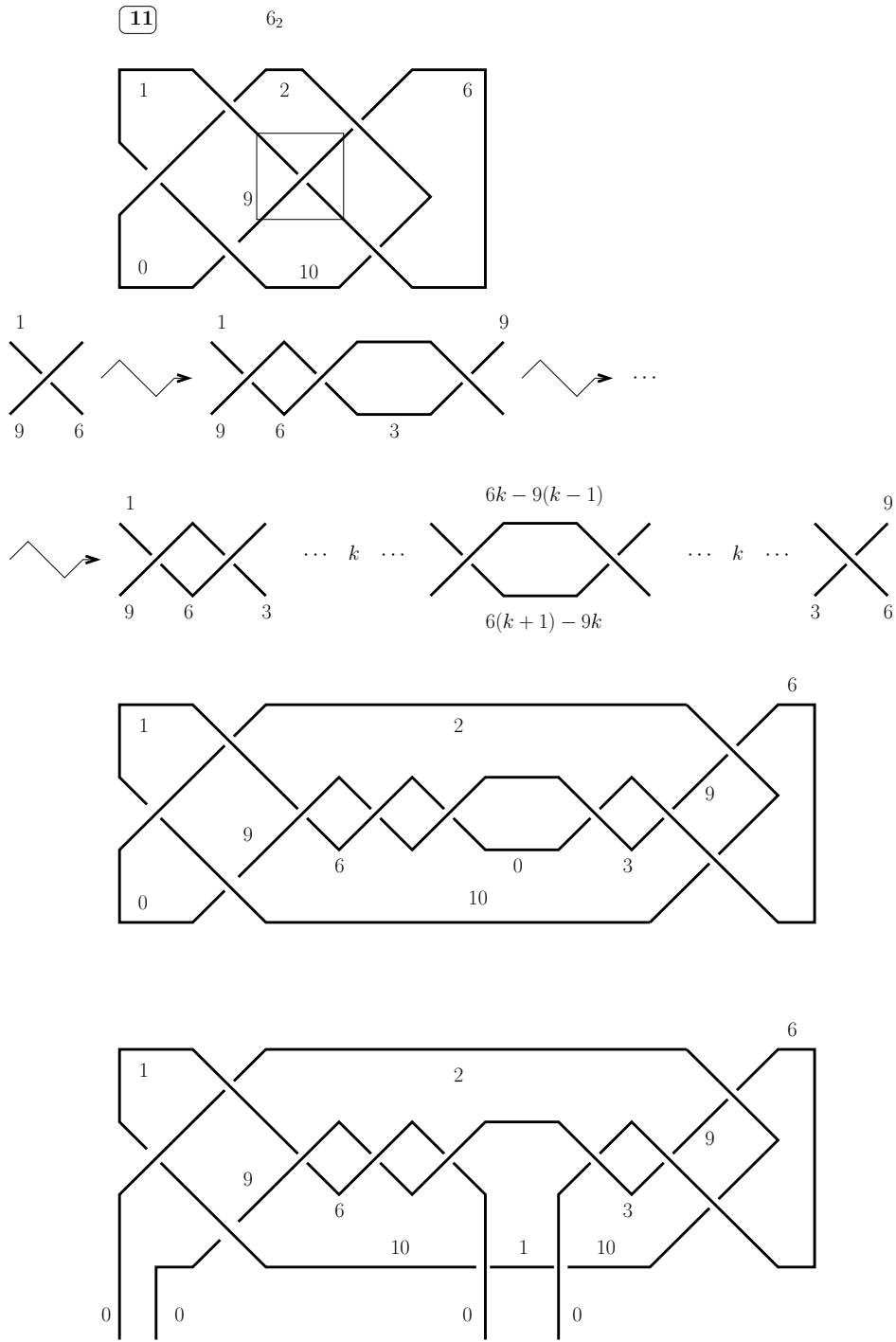}}}
	\caption{Top: knot $6_2$ non-trivially colored mod $11$; the boxed crossing is worked out below via type II Reidemeister moves and later inserted into the diagram. In the bottom we have a $2$-tangle whose arcs are knotted.}\label{fig:6-2}
\end{figure}

Here is (Corollary \ref{cor:rat-persist-tangle}) another systematic way of producing persistent tangles in the spirit of Corollary \ref{thm:persistent-2-tangles}.

\begin{cor}\label{cor:rat-persist-tangle}
Given any rational tangle, $T$, tangle addition of it to $T^{\star}$, its mirror image, produces a persistent tangle.
\end{cor}
\begin{proof}
See Figure \ref{fig:T-Tstar} for an illustration.
\begin{figure}[!ht]
	\psfrag{0}{\huge$0$}
	\psfrag{1}{\huge$1$}
	\psfrag{2}{\huge$2$}
	\psfrag{3}{\huge$3$}
	\psfrag{4}{\huge$4$}
	\psfrag{-6}{\huge$-6$}
	\psfrag{-3}{\huge$-3$}
	\psfrag{=}{\huge$=$}
	\psfrag{+}{\huge$+$}
	\psfrag{T+}{\Huge$T^{\star}$}
	\psfrag{T+T+}{\Huge$T+T^{\star}$}
	\psfrag{T}{\Huge$T$}
	\psfrag{-6=1}{\huge$-6=_{7} 1$}
	\psfrag{p-1}{\huge$p-1$}
	\centerline{\scalebox{.4}{\includegraphics{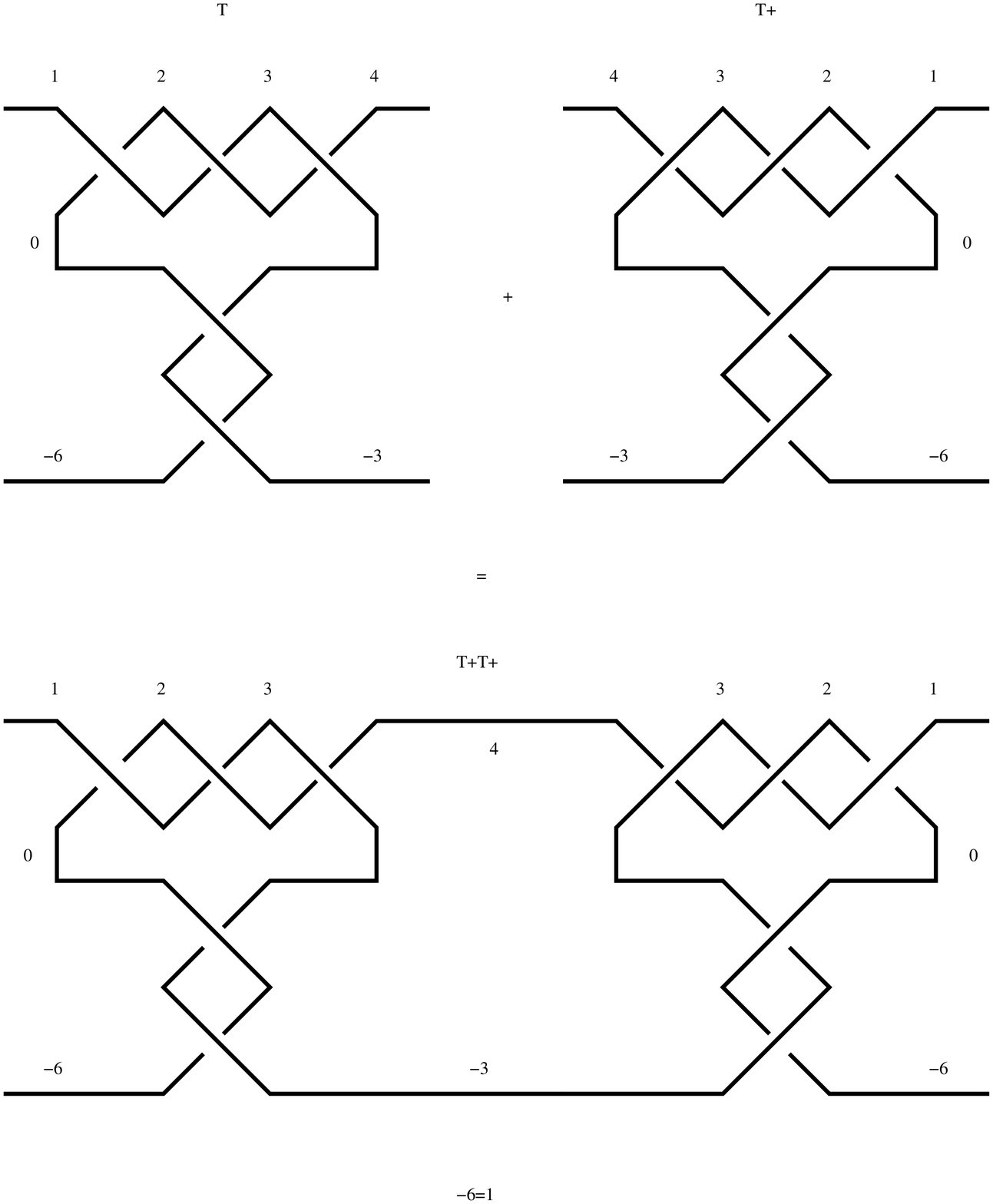}}}
	\caption{A persistent tangle. The non-trivial coloring at issue is by the dihedral quandle of order $7$.}\label{fig:T-Tstar}
\end{figure}
\end{proof}

\begin{definition}\label{def:rationallyirred}
A tangle is irreducible if no ambient isotopy plus adjacent end twisting can make it into a tangle with fewer crossings, and it is neither an infinity tangle nor a zero tangle.

Note that rational knots and links are by definition closures of rational tangles. Rational tangles are those tangles that are rationally reducible. But non-rational tangles can sometimes reduce to rational knots (we have specific examples in the paper).
\end{definition}

\begin{definition}\label{def:irred}
A tangle $T$ is said to be {\bf irreducible} if it is rationally irreducible, without local knots, and if whenever the numerator $N(T)$ or denominator $D(T)$ closure has one component, then it is a non-trivial knot.
\end{definition}

\begin{conj}
An irreducible tangle is a persistent tangle.
\end{conj}

A small example of an irreducible tangle whose persistent we are not yet able to prove, is given in Figure \ref{fig:irred-2-tangle}. Many of the persistent tangles already discussed in this paper are irreducible. However, a proof of this conjecture would definitely require techniques beyond the coloring approach of the present paper. For example, the tangle in Figure \ref{fig:non-rat-tangle} is shown, in that figure, to admit no non-trivial coloring that constantly labels all of its tangle ends.

\begin{figure}[!ht]
	\psfrag{Fig-8}{\huge$\text{Figure-$8$ knot, $N(T)$}$}
	\psfrag{T}{\huge$T$}
	\psfrag{DT}{\huge$D(T), \text{ Trefoil knot}$}
	\psfrag{3}{\huge$3$}
	\psfrag{4}{\huge$4$}
	\psfrag{5}{\huge$5$}
	\psfrag{6}{\huge$6$}
	\psfrag{7}{\huge$7$}
	\psfrag{8}{\huge$8$}
	\psfrag{9}{\huge$9$}
	\psfrag{10}{\huge$10$}
	\psfrag{6k-9(k-1)}{\huge$6k-9(k-1)$}
	\psfrag{6(k+1)-9k}{\huge$6(k+1)-9k$}
	\psfrag{11}{\huge$\mathbf{11}$}
	\psfrag{6-2}{\huge$6_2$}
	\psfrag{k}{\huge$k$}
	\psfrag{...}{\huge$\cdots$}
	\centerline{\scalebox{.4}{\includegraphics{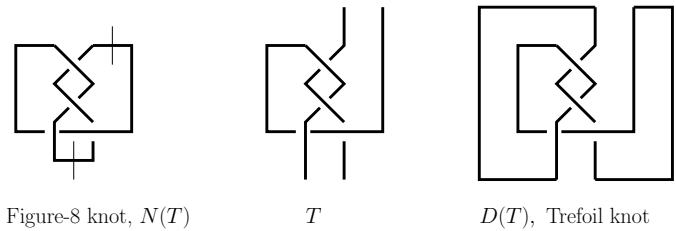}}}
	\caption{Left: Figure-$8$ knot; numerator closure of tangle $T$ (notation, $N(T)$) in the middle. Middle: tangle $T$. Right: denominator closure of $T$, $D(T)$ (it is the trefoil).  Remark: both $N(T)$ and $D(T)$ are non-trivial knots. $T$ is thus an example of an irreducible tangle. However, if we color the slashed arcs with the same color then this assignment will necessarily extend to a trivial coloring. Therefore, coloring arguments are not enough to prove this tangle is persistent - should it be persistent.}\label{fig:irred-2-tangle}
\end{figure}

\begin{figure}[!ht]
	\psfrag{Fig-8}{\huge$\text{Figure-$8$ knot, $N(T)$}$}
	\psfrag{T}{\huge$T$}
	\psfrag{DT}{\huge$D(T), \text{ Trefoil knot}$}
	\psfrag{2}{\huge$2$}
	\psfrag{0}{\huge$0$}
	\psfrag{1}{\huge$1$}
	\psfrag{6}{\huge$6$}
	\psfrag{7}{\huge$7$}
	\psfrag{8}{\huge$8$}
	\psfrag{9}{\huge$9$}
	\psfrag{10}{\huge$10$}
	\psfrag{6k-9(k-1)}{\huge$6k-9(k-1)$}
	\psfrag{6(k+1)-9k}{\huge$6(k+1)-9k$}
	\psfrag{11}{\huge$\mathbf{11}$}
	\psfrag{6-2}{\huge$6_2$}
	\psfrag{k}{\huge$k$}
	\psfrag{...}{\huge$\cdots$}
	\centerline{\scalebox{.4}{\includegraphics{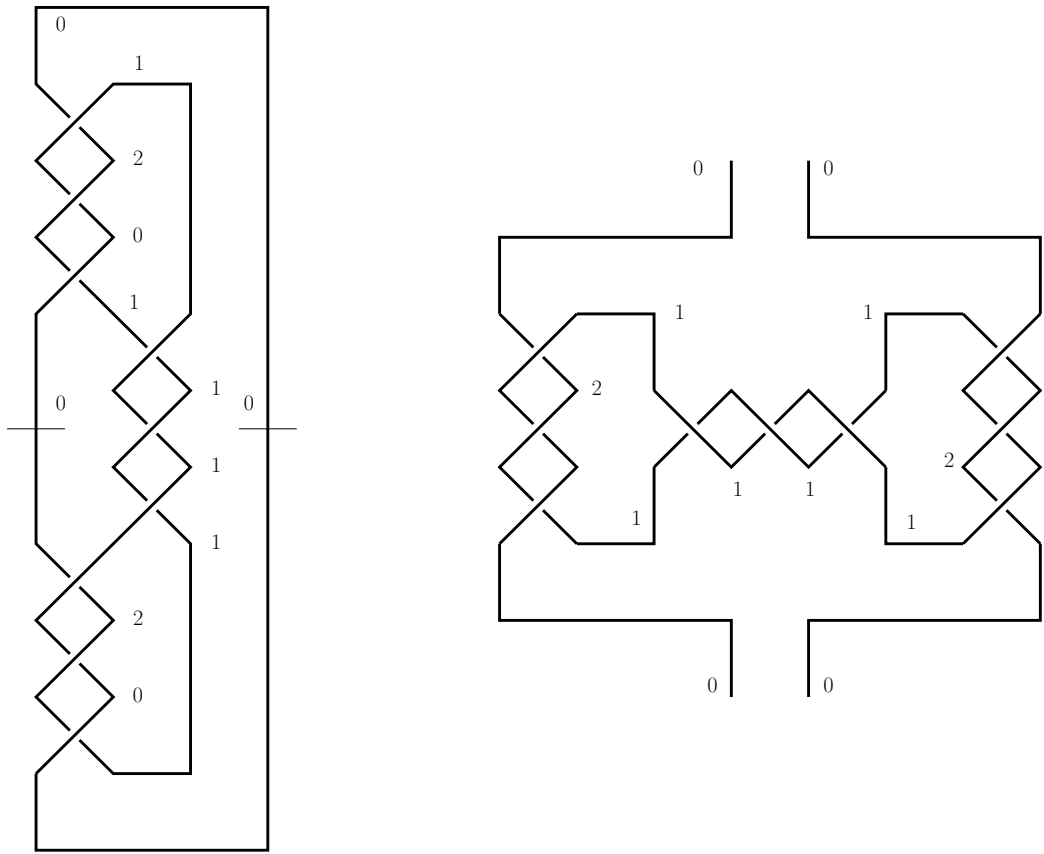}}}
	\caption{Cutting open a rational knot to give a non-rational tangle that is persistent and reduces to the Hercules tangle.}\label{fig:irred-persistent}
\end{figure}

\begin{figure}[!ht]
	\psfrag{a}{\huge$a$}
	\psfrag{b}{\huge$b$}
	\psfrag{c=b*a}{\huge$c=b\ast a$}
	\psfrag{e=b*a}{\huge$e=b\ast a$}
	\psfrag{c=bn*a}{\huge$c=b\bar{\ast} a$}
	\psfrag{e=bn*a}{\huge$e=b\bar{\ast} a$}
	\psfrag{d}{\huge$d$}
	\centerline{\scalebox{.4}{\includegraphics{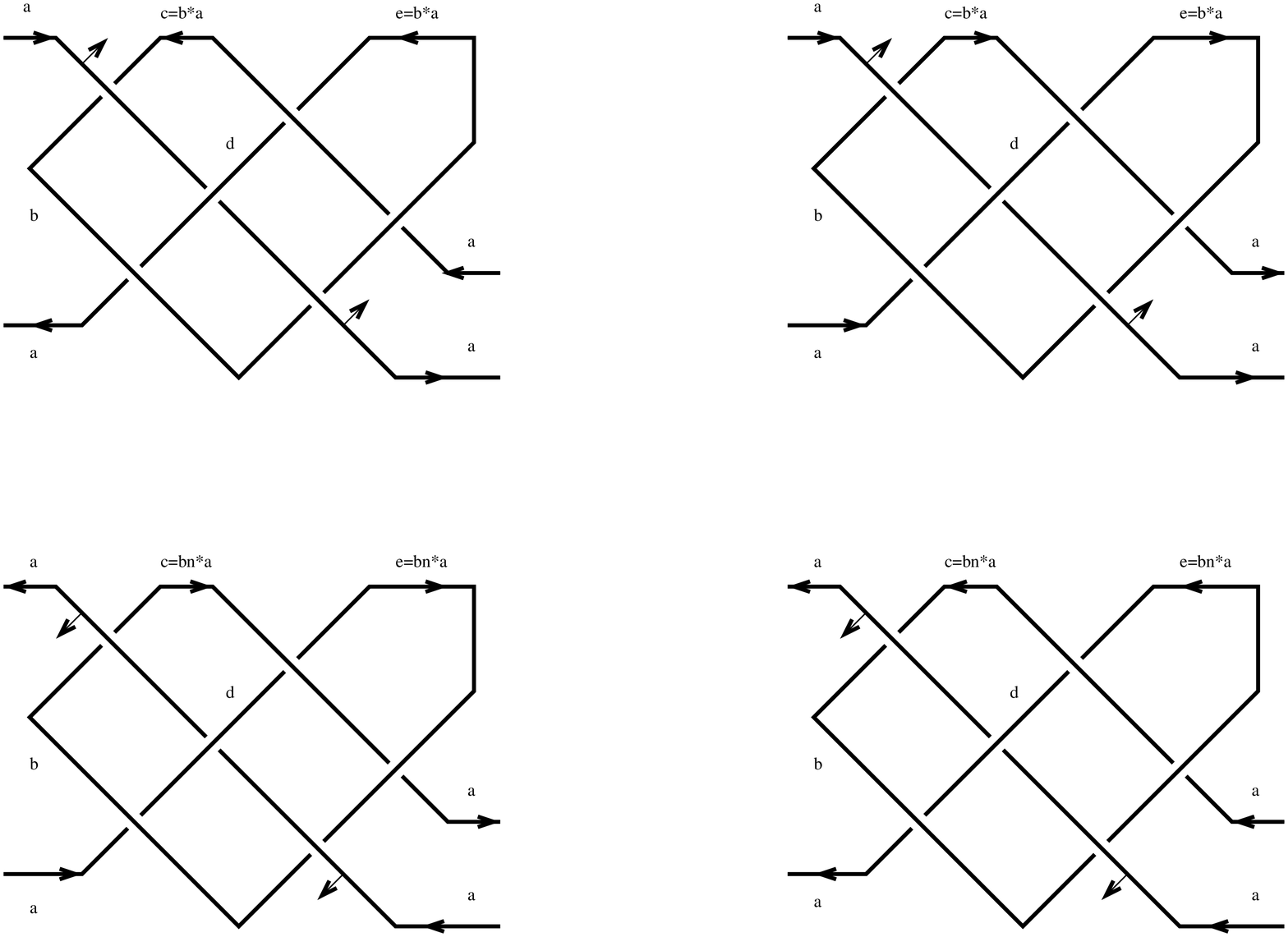}}}
	\caption{The same $2$-tangle with the four different assignments of orientations to its arcs; the configurations on the left-hand side are prepared for denominator closure whereas the ones on the right-hand side are prepared for numerator closure, for instance. We try to equip them with a non-trivial coloring keeping the start- and end-points with the same color, $a$. We arrive at an inconsistency: $e=c$ in each instance which implies the colorings have to be trivial, again in each instance.}\label{fig:non-rat-tangle}
\end{figure}

In Figure \ref{fig:hercules-over} yet another instance. It is a non-trivial knot since it features Krebes original tangle although in a more elaborate way.

\begin{figure}[!ht]
	\centerline{\scalebox{.4}{\includegraphics{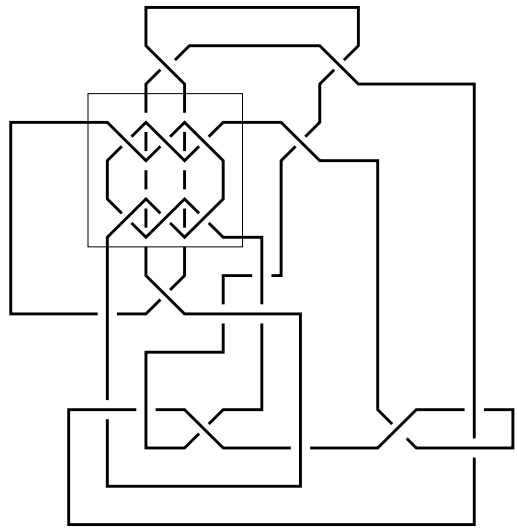}}}
	\caption{A non-trivial knot since the diagram features Krebes original tangle. A subtlety: Krebes original tangle lies over other portions of the diagram.}\label{fig:hercules-over}
\end{figure}


\begin{thebibliography}{99}



\bibitem{SMAbernathy}
        S. M. Abernathy, \emph{On Krebe's tangle}, Topology
Appl., {\bf160} (2013), 1379--1383



\bibitem{DJoyce}
        D. Joyce, \emph{A classifying invariant of knots, the knot quandle}, J. Pure
Appl. Alg., {\bf23} (1982), 37--65


\bibitem{HenrichKauffman}
A. Henrich and L. H. Kauffman, \emph{Persistent tangles are irrational} (to appear).





\bibitem{lhKauffman}
        L. H. Kauffman, \emph{Knots and physics}, 4th edition,
        Series on Knots and Everything {\bf53},
        World Scientific Publishing Co., 2013

\bibitem{KauffmanGoldman}
L. H. Kauffman and J. Goldman, \emph{Knots Tangles and Electrical Networks}, Adv. in Appl. Math., {\bf14} (1993), 267-306.



\bibitem{LambropoulouKauffman}
L. H. Kauffman and S. Lambropoulou, \emph{Hard unknots and collapsing tangles.} in ``Introductory Lectures in
Knot Theory", K\&E Series Vol. 46, edited by Kauffman, Lambropoulou, Jablan and Przytycki,
World Scientiic 2011, pp. 187 - 247.



\bibitem{Krebes} D. Krebes, \emph{An obstruction to embedding 4-tangles in links}, J. Knot Theory Ramifications {\bf 08} (1999),  no. 03, 321--352.


\bibitem{KrebesSilverWilliams} D. Krebes, D. S. Silver, S. G. Williams \emph{Persistent invariants of tangles}, J. Knot Theory Ramifications {\bf 09} (2000),  no. 04, 471--477.



\bibitem{Lickorish}
       W. B. R. Lickorish, \emph{An introduction to knot theory}, Graduate Texts in Mathematics
         {\bf175}, Springer Verlag, New York (1997).



\bibitem{pLopes}
        P. Lopes, \emph{Quandles at finite temperatures I},
        J. Knot Theory Ramifications, {\bf12} (2003), no. 2, 159--186



\bibitem{SMatveev}
       S. Matveev, \emph{Distributive groupoids in knot theory},
       Math. USSR Sbornik {\bf 47} (1984), no. 1, 73--83.



\bibitem{PrzytyckiSilverWilliams}
J. H. Przytycki, D. S. Silver, S. G. Williams, \emph{3-manifolds, tangles, and persistent invariants}, Math. Proc. Cambridge Philos. Soc. {\bf139} (2005), 291--306.



\bibitem{Ruberman} D. Ruberman, \emph{Embedding tangles in links}, J. Knot Theory Ramifications {\bf 09} (2000),  no. 04, 523--530.


\bibitem{SilverWilliams}
D. S. Silver, S. G. Williams, \emph{Virtual tangles and a theorem of Krebes}, J. Knot Theory Ramifications {\bf 08} (1999),  no. 07, 941--945.



\bibitem{SilverWilliams2}
D. S. Silver and S.G. Williams, \emph{Tangles and links: a view with trees}, J. Knot Theory Ramifications {\bf 27}, no. 12, October 2018.


\end{thebibliography}
\end{document}